\documentclass[11pt]{amsart} \usepackage{amssymb}
\usepackage{paralist} \usepackage{graphicx} \usepackage{hyperref}
\usepackage[all]{xy}

\title[Coupling solutions of BGG-equations]{Coupling solutions of BGG-equations in conformal spin geometry}
\author{Matthias Hammerl} 
\email{matthias.hammerl@univie.ac.at}
\address{Faculty of Mathematics,
University of Vienna, Nordbergstra\ss e~15, A--1090 Wien, Austria}
\date{\today}

\subjclass[2000]{ 35N10, 53A30, 53B15} 
\keywords{overdetermined systems, conformal spin geometry,twistor spinors,
conformal Killing fields, conformal Killing forms, almost Einstein
scales}

\def\mb{\mathbf} 

\def\mr{\mathrm}
\def\mc{\mathcal}

\newtheorem{thm}{Theorem}[section]
\newtheorem{prop}[thm]{Proposition}
\newtheorem*{prop*}{Proposition}
\newtheorem*{thm*}{Theorem}
\newtheorem{lem}[thm]{Lemma}
\newtheorem*{lem*}{Lemma}

\newtheorem*{cor*}{Corollary}

\theoremstyle{definition}
\newtheorem{defn}[thm]{Definition}
\newtheorem*{defn*}{Definition}
\theoremstyle{remark}
\newtheorem{rem}[thm]{Remark}
\newtheorem*{remk*}{Remark}

\def\mb{\mathbf}

\def\mr{\mathrm}
\def\mc{\mathcal}

\newlength{\equwidth}
\newcommand{\crd}{\mbox{$                                     
\begin{picture}(9,8)(1.6,0.15)
\put(1,0.2){\mbox{$ D \hspace{-7.8pt} /$}}
\end{picture}$}}
\def\dirac{\crd}

\DeclareMathOperator{\PP}{P}

\DeclareMathOperator{\aEs}{aEs}
\DeclareMathOperator{\cKf}{cKf}

\DeclareMathOperator{\Hom}{Hom}

\def\A{\mathcal A}

\def\X{\mathfrak X}

\def\ga{\gamma}
\def\de{\delta}

\def\rh{\rho}

\def\si{\sigma}

\def\om{\omega}
\def\Ga{\Gamma}
\def\De{\Delta}
\def\Th{\Theta}
\def\La{\Lambda}
\def\Si{\Sigma}
\def\Ph{\Phi}

\def\Om{\Omega}

\def\pa{\partial}
\def\t{\otimes}

\def\delstar{\pa^*}

\def\goesto{\rightarrow}
\def\lar{\leftarrow}

\def\bdot{\mathrm{\bullet}}
\def\na{\mathrm{\nabla}}

\def\dcov{\mathrm{d^{\nabla}}}

\def\im{\mathrm{im}\ }

\def\G{\mathcal{G}}

\DeclareMathOperator{\End}{End}
\DeclareMathOperator{\GL}{GL}
\DeclareMathOperator{\Spin}{Spin}
\DeclareMathOperator{\CSpin}{CSpin}

\def\ti{\tilde}

\def\rr{\ensuremath{\mathbb{R}}}

\def\CC{\ensuremath{\mathcal{C}}}
\def\ZZ{\ensuremath{\mathcal{Z}}}
\def\BB{\ensuremath{\mathcal{B}}}
\def\HH{\ensuremath{\mathcal{H}}}

\def\VV{\ensuremath{\mathcal{V}}}

\def\Hom{\ensuremath{\mathrm{Hom}}}

\def\GL{\ensuremath{\mathrm{GL}}}

\def\g{\ensuremath{\mathfrak{g}}}

\def\ce{\ensuremath{\mathcal{E}}}

\DeclareMathOperator{\sym}{sym}

\newcommand{\lpl}{
  \mbox{$
  \begin{picture}(12.7,8)(-.5,-1)
  \put(2,0.2){$+$}
  \put(6.2,2.8){\oval(8,8)[l]}
  \end{picture}$}}

\begin{document}

\maketitle

\begin{abstract}
BGG-equations are geometric overdetermined systems of
PDEs on parabolic geometries. Normal solutions of BGG-equations
are particularly interesting and we give a simple formula
for the necessary and sufficient additional integrability conditions on
a solution. We then discuss a procedure for coupling
known solutions of BGG-equations to produce new ones.
Employing a suitable calculus
for conformal spin structures  this yields 
explicit coupling formulas and conditions between
almost Einstein scales, conformal Killing forms and twistor spinors. 
Finally we discuss a class of generic twistor spinors
that provides an invariant decomposition of conformal Killing
fields.
\end{abstract}

\section{Introduction}\label{secintr}
  Let $M$\ be a smooth manifold and $(\G\goesto M,\om)$\ a parabolic
  geometry of type $(G,P)$. Here $G$\ a is semi-simple Lie group,  $P\subset G$
  a parabolic subgroup and $\om\in\Om^1(\G,\g)$\ the Cartan connection
form of the geometry with values in the Lie algebra $\g$\ of $G$.
Geometries of interest could for instance be projective structures, conformal structures or CR-structures. The Cartan connection form $\om$\ generalizes
the properties of the Maurer-Cartan form $\om^{MC}\in\Om(G,\g)$ to
the curved setting, \cite{cap-slovak-par}.

We are interested in overdetermined operators on such geometries
which appear as the first operators in the BGG-sequence
  \[\HH_0\overset{\Th_0}{\goesto}\HH_1\overset{\Th_1}{\goesto}\cdots\overset{\Th_{n-1}}{\goesto}\HH_n\]
of natural differential operators as constructed in \cite{bgg-2001}\
and then presented in a simplified form in \cite{bgg-calderbank-diemer}.

The study of the BGG-sequence and in particular
of the first BGG-operators, and the BGG-equations
these describe, has seen much interest in recent years.
It has been shown that the infinitesimal symmetries of a parabolic geometry
can be described by a BGG-equation, \cite{cap-infinitaut},
and that the BGG-equations are always of finite type, \cite{prolong,hsss}.
Moreover, solutions of BGG-equations have been shown to appear naturally as characterizing
objects of Fefferman-type spaces, \cite{cap-gover-cr,
mrh-sag-rank2,mrh-sag-twistors}.

Since the BGG-machinery that describes
these equations starts from a uniform algebraic setting it is
also reasonable to ask whether this construction can be used
to obtain relations between solutions of different BGG-equations.
An abstract formulation of this question was introduced in
\cite{bgg-calderbank-diemer} via the notion of cup product.
Explicit calculations and results were achieved
in \cite{gover-silhan-2006} under the name of helicity
raising and lowering for conformal Killing forms.
While not mentioning the BGG-machinery
there, it is clear that this kind of construction is possible
for certain classes of BGG-operators on parabolic geometries.

\subsection{Outline}

In section \ref{secbgg} we briefly review the construction
of the BGG-operators and the prolongation connection. We discuss
normality of a solution and give a simple formula which provides
the necessary and sufficient equations. We then introduce
coupling maps for solutions of BGG-equations. For $|1|$-graded
parabolic geometries and coupling maps where the target space
is the domain of a BGG-operator of first order
we give necessary and sufficient coupling conditions.

In section \ref{secconf} the coupling procedure of section \ref{secbgg}
will be applied to conformal spin structures. We briefly introduce
these structures and then discuss several interesting first
BGG-operators: those governing almost Einstein scales, conformal Killing
forms and generic twistor spinors. We then derive explicit
coupling maps and conditions for these objects. 

The formulas obtained for coupling with twistor spinors are
particularly interesting when this spinor is generic in a
suitable way. We show in section \ref{subsecgen} that every
generic twistor spinor gives rise to a natural decomposition
of conformal Killing fields. For signatures 
$(2,3)$\ and $(3,3)$, such twistor spinors have been constructed
in \cite{mrh-sag-twistors}\ and we discuss this result from
the viewpoint of coupling maps.

\subsubsection*{\textbf{Acknowledgments}}
Discussions with  Andreas \v{C}ap, A. Rod Gover, Felipe Leitner, Katja Sagerschnig and Josef \v{S}ilhan have been very valuable. 
The author also gladly acknowledges support from
 project P 19500-N13 of the "Fonds zur
F\"{o}rderung der wissenschaftlichen Forschung" (FWF)
and from START prize project Y377 of the Austrian Federal Ministry of Science and Research.

\section{BGG-equations, normality and coupling}\label{secbgg}
We begin with a very brief introduction of the necessary tractor calculus
for parabolic geometries. For more background we refer to \cite{cap-gover-tractor,cap-slovak-par}.
\subsection{Tractor bundles}
  For every irreducible $G$-representation $V$\
  one associates the tractor bundle $\VV=\G\times_P V$.
  It is well known, cf. 
  that $\mc{V}$ carries its canonical
    tractor connection, denoted by $\na=\na^V$.
  By forming the exterior covariant derivative $\dcov$\ of $\na$\ on
  $\mc{V}$-valued differential forms
  this gives rise to the sequence
  \[\CC_0\overset{\na}{\goesto}\CC_1\overset{d^{\na}}{\goesto}\CC_2\overset{d^{\na}}{\goesto}\cdots\]
  on the chain spaces $\CC_k=\Om^k(M,\VV)$.

    Moreover, one has the (algebraic) Kostant\ co-differential $\delstar: \CC_{k+1}  \goesto \CC_k$,\ $\delstar\circ\delstar=0,$\ which yields the complex
  \[\CC_0\overset{\delstar}{\lar}\CC_1\overset{\delstar}{\lar}\CC_2\overset{\delstar}{\lar}\cdots .\]
  This complex gives rise to spaces $\ZZ_k=\ker\delstar$\ of cycles,
  borders $\BB_k=\im\delstar$\ and homologies $\HH_k=\ZZ_k/\BB_k$.
  The canonical surjections are denoted $\Pi_k:\ZZ_k\goesto\HH_k$.

\subsection{BGG-operators and the prolongation connection}
The BGG-machinery of \cite{bgg-2001}
is based on canonical differential splitting operators $L_k:\Ga(\HH_k)\goesto \Ga(\ZZ_k)$: 
A section $s\in\Ga(\ZZ_k)$\ is of the form $L_k\si,\si\in\Ga(\HH_k)$\ if and
only if $d^{\na}s\in\ker\delstar$. This uniquely defines the operators
$L_k$.

Now, given a section $\si\in\Ga(\HH_k)$, one can form
$\dcov(L_k\si)\in\Om^{k+1}(M,\mc{V})$, which by assumption on $L_k$\ is a section
of $\ZZ_{k+1}$, and can therefore be projected to $\Ga(\HH_{k+1})$.
The composition $\Th_{k}:=\Pi_{k+1}\circ\dcov\circ L_{k}$\
is the $k+1$-st BGG-operator.

For $k=0$\ one obtains the first BGG-operator $\Th_0=\Pi_1\circ\na\circ L_0$, $\Th_0:\Ga(\HH_0)\goesto\Ga(\HH_1)$, which is an overdetermined operator.
One does in fact have that the system $\si\in\Ga(\HH_0),\ \Th_0(\si)=0$\
is of finite type:  In \cite{hsss}
  a natural modification $\ti\na=\na+\Psi$\ with
  $\Psi\in\Om^1(M,\End(\mc{V}))$\ was constructed which
  has the following property:
\begin{prop}[\cite{hsss}]\label{prop-hsss}
  The solutions $\si\in\Ga(\mc{H}_0)$\ of $\Th_0(\si)=0$\ are
  in $1:1$-correspondence with the $\ti\na$-parallel
  sections of $\mc{V}$. This isomorphism is realized with
  the first BGG-splitting operator $L_0:\Ga(\HH_0)\goesto\Ga(\mc{V})$\
  and the canonical projection $\Pi_0:\Ga(\mc{V})\goesto\Ga(\HH_0)$.
\end{prop}
  We call $\ti\na$ the prolongation connection of $\Th_0$\ since
  the equation $s\in\Ga(\mc{V}), \ti\na s=0$\ is the
  prolongation of the system $\si\in\Ga(\mc{H}_0),\Th_0(\si)=0$.

\subsection{Normal solutions}
  If a section $s\in\Ga(\mc{V})$\ is $\na$-parallel, we automatically
  have that $s\in\im L_0$, since $\delstar(\na s)$\ is vanishes
  trivially, and so $s=L_0 \si$\ with $\si=\Pi_0 s$.
  Then $\Th_0 \si=\Pi_1(\na L_0 \si)=\Pi_1(\na s)=0$.
  We say that those $\si\in\ker\Th_0$\ that
  satisfy $\na L_0 \si=0$\ are the \emph{normal\ solutions}\
  of $\Th_0(\si)=0$. If the geometry is flat, all solutions are
  normal.

  Now, if $\si\in\Ga(\HH_0)$\ is an arbitrary solution of $\Th_0(\si)=0$,
  then equivalently, with $s=L_0\si$\ and $\ti\na$\ the
  prolongation connection,
  \begin{align*}
    0=\ti\na s=\na s+\Psi s.
  \end{align*}
  Thus $\Psi L_0\si$\ is the obstruction against
  normality of $\si$.
  However, it turns out that determining normality
  of a solution does not need computation of $\Psi$, which is
  always possible but depends on a procedure that involves
  one iteration for every filtration component of the tractor bundle,
  \cite{hsss}.

To state the simple criteria for normality we need to introduce 
the curvature of the Cartan connection form $\om$.
It is defined as $K(\xi,\xi')=d\om(\xi,\xi')+[\om(\xi),\om(\xi')]$\ for
$\xi,\xi'\in\X(\G)$. This determines a $2$-form on $\G$\ with values in $\g$.
By forming the \emph{adjoint tractor bundle} $\mc{A}:=\G\times_P\g$\
and using horizontality and $P$-equivariancy of $K$, we can
equivalently regard it as $K\in\Om^2(M,\A M)$.
Now since $V$\ is a $G$-representation, the 
Lie algebraic action of $\g$\ on $V$\ yields 
action of $\mc{A}=\G\times_P\g$\ on $\mc{V}=\G\times_P V$,
which we denote via $\bdot:\mc{A}\goesto\End(\mc{V})$.

\begin{prop}\label{propcoupling}
  A solution $\si$\ of $\Th_0(\si)=0$\ is
  normal if and only if $\delstar\bigl(K\bdot (L_0 \si) \bigr)=0$.
\end{prop}
\begin{proof}
    Since $\na$\ is the natural connection induced by $\om$\ on $\mc{V}$
  the curvature of $\na$\ is $R=K\bdot\in\Om^2(M,\End(\mc{V}))$.

Denote $s=L_0\si$. Now if $\delstar(K\bdot s)=0$, then since $R=d^{\na}\circ\na$
we have $d^{\na}(\na s)\in\ker \delstar$. Thus
\[
\na s=L_1(\Pi_1(s))=L_1(\Th_0 \si)=0.
\]

The converse is clear, since if $\na s=0$, also
$0=\dcov\circ\na s=Rs=K\bdot s$, 
and then $0=\delstar(K\bdot s)=\delstar(K\bdot L_0\si)$.
\end{proof}

We now discuss a procedure for obtaining new solutions from
known ones. This will be particularly simple for normal solutions,
but milder conditions on the solutions will be sufficient for
interesting classes of equations.

\subsection{Coupling maps}\label{subseccoupl}
  Let $V,V'$\ and $W$\ be $G$ representations
  and $C:V\times V'\goesto W$\ be a $G$-equivariant
  bilinear map. The corresponding tractor map
  is denoted  $\mb{C}:\mc{V}\times\mc{V'}\goesto\mc{W}.$

  It induces the (differential) coupling map $\mb{c}:\Ga(\HH_0^V)\times\Ga(\HH_0^{V'})\goesto\Ga(\HH_1^W)$,
  \begin{align*}
        (\si,\si')\mapsto \Pi_0^W\bigl(\mb{C}(L_0^V(\si),L_0^{V'}(\si'))\bigr).
  \end{align*}

  Since $\mb{C}:\mc{V}\times\mc{V'}\goesto\mc{W}$\ is algebraic and
  natural, we have that for all $s\in\Ga(\mc{V}), s'\in\Ga(\mc{V'})$,
  \begin{align}\label{Ccompat}
    \na^{W}\mb{C}(s,s')=\mb{C}(\na^Vs,s')+\mb{C}(s,\na^{V'}s').
  \end{align}
  In particular, if $\si\in\HH_0^V$\ and $\si'\in\HH_0^{V'}$\ are
  normal solutions of $\Th_0^V$\ resp. $\Th_0^{V'}$, then
  $\eta:=\mb{c}(\si,\si')$\ is a normal solution of $\Th_0^W$.

 \subsubsection{Coupling for $|1|$-graded parabolic geometries with $\Th_0^W$ of   first order}
  The operators $\Th_0^V$\ and $\Th_0^{V'}$\ have prolongation connections
  $\ti\na^V=\na^V+\Psi^V$, $\ti\na^{V'}=\na^{V'}+\Psi^{V'}$. We write
  $s=L_0^V \si,s'=L_0^{V'}\si'$.

  By definition, $t:=\mb{C}(s,s')$\ is a 
  lift of $\eta=\mb{c}(\si,\si')=\Pi_0^W(t)$, but
  one doesn't necessarily have $t=L_0^W\eta$:
  $\na^W t\in\CC_1^W=\Om^1(M,\mc{W})$\ need not
  lie in $\ZZ_1^W=\ker\delstar$. We will circumvent this
  problem by building a canonical extension of $\Pi_1^W:\ZZ_1^W\goesto\HH_1^W$\ to
  a map $\Pi_{1,\odot}^W:\CC_1^W\goesto\HH_1^W$.

  For this one uses the natural filtration
  of $W$\ that is induced 
  by the $P$-representation on $W$. The
  largest filtration component is just $\overline{W}:=\im\delstar\subset W$.
  The parabolic group $P$\ has a Levi factor $G_0\subset P$\
  and $W_0:=W/{\overline{W}}$\ is a well defined $G_0$-representation.
  On the level of associated bundles one
  obtains a semidirect composition series
  $\mc{W}=\mc{W}_0\lpl \overline{\mc{W}}$
  and this induces the semidirect composition series
  \begin{align*}
    \CC_1^W=T^*M \t \mc{W}=T^*M \t \mc{W}_0\lpl T^*M \t \overline{\mc{W}}.
  \end{align*}
  In particular, we have a canonical surjection
  $\Pi^{\mc{W}_0}_1:\CC_1^W\goesto T^*M\t\mc{W}_0$.

  The fact that
  $\Th_0^W=\Pi_1^W\circ\na\circ L_0^W$\ is an operator of first order
  is equivalent to $\HH_0$\ not depending on $\Om^1(M,\overline{\mc{W}})$,
  i.e., $\Pi_1^W(\ker\delstar\cap T^*M\t \overline{\mc{W}})=\{0\}$, \cite{prolong}. Then $\HH_0\subset T^*M\t \mc{W}_0$\ is the highest
  weight part with respect to the $G_0$-structure.
  So composing the projection to the highest weight part
  with $\Pi^{\mc{W}_0}_1$\ yields a map
  $\Pi_{1,\odot}^W:\CC_1^W\goesto \HH_1$
  with the property that its restriction to $\ZZ_1^W$\ is just $\Pi_1^W$.
  The operator $\Th_0$\ can now be written
  \begin{align}\label{th0form}
    \Th_0^W(\eta)=\Pi_{1,\odot}^W(\na^W t).
  \end{align}  

  \begin{prop}\label{propTh0W}
  For $\si\in\ker\Th_0^V,\si'\in\ker\Th_0^{V'}$\ and $\eta=\mb{c}(\si,\si')$\ one has
  $\Th_0^W(\eta)=-\Pi^W_{1,\odot}\bigl(\mb{C}(\Psi^V s,s')+\mb{C}(s,\Psi^{V'}s')\bigr).$
  \end{prop}
  \begin{proof}
      For $\si\in\ker\Th_0^V$\ we have equivalently
  that $s=L_0^V$\ satisfies
  \begin{align*}
    0=\ti\na^V s=\na^Vs+\Psi^Vs,
  \end{align*}
  so $\na^Vs=-\Psi^V s$,
  and analogously for $\si'\in\ker\Th_0^{V'}$.
  Therefore
   $\na^Wt=\na^W\mb{C}(s,s')=\mb{C}(\na^V s,s')+\mb{C}(s,\na^{V'} s')= 
    -\mb{C}(\Psi^V s,s')-\mb{C}(s,\Psi^{V'}s').$
Thus, using \eqref{th0form}, this proves the claim.
  \end{proof}
  In particular, this yields necessary and
  sufficient coupling conditions.  

\section{Coupling in conformal spin geometry}\label{secconf}

  A conformal spin structure of signature $(p,q)$\
  on an $n=p+q$-manifold $M$\ is a reduction
  of structure group of $TM$\ from $\GL(n)$\ to
  $\CSpin(p,q)=\rr_+\times\Spin(p,q)$.
  This induces a conformal class $\mc{C}$ of pseudo-Riemannian
  signature $(p,q)$-metrics on $M$.
  The associated bundle to the spin representation
  $\De^{p,q}$\ of $\CSpin(p,q)$ with $\rr_+$\ acting trivially
  is the (unweighted) conformal spin bundle $\mc{S}$.

  We will often employ the conformal density bundles 
  $\ce[w]$, $w\in\rr$, which are associated to the $1$-dimensional
  $\rr_+$ representations $c\mapsto c^w$. We also employ abstract
  index notation $\ce_a=\Ga(T^*M)=\Om^1(M),\ce^a=\Ga(TM)=\X(M)$\
  with multiple indices denoting tensor products, e.g.
  $\ce_{ab}=\Ga(T^*M\t T^*M)$.

The curvature quantities of the conformal
structure $\mc{C}$\ are computed with respect
to a $g\in\mc{C}$.
The symmetric $2$-tensor
\begin{align*}
  \mr{P}_g:=\frac{1}{n-2}(\mr{Ric}_g-\frac{\mr{Sc}_g}{2(n-1)}g)
\end{align*}
is the \emph{Schouten\ tensor}; this is a trace modification
of the Ricci curvature $\mr{Ric}_g$\ by a multiple
of the scalar curvature $\mr{Sc}_g$. The trace of
the Schouten tensor is denoted $J_g=g^{pq}\PP_{pq}$.
We will omit the
subscripts $g$\ hereafter when giving a formula
with respect to some $g\in\mc{C}$.

The complete obstruction against
conformal flatness of $(M,\mc{C})$\ with, $\mc{C}$\ having
signature $p+q\geq 3$, is the \emph{Weyl curvature}
\begin{align*}
  C_{ab\; d}^{\;\;\; c}:=R_{ab\; d}^{\;\;\; c}-2\de_{[a}^c\PP_{b]d}+2g_{d[a}\PP_{b]}^{\; c},
\end{align*}
where $R$\ is the Riemannian curvature tensor of $D$\
and indices between square brackets are skewed over, (cf. e.g. \cite{eastwood-notes-conformal}.)

 A conformal spin structure of signature $(p,q)$\
  is equivalently described by a parabolic geometry of type $(\Spin(p+1,q+1),P)$,
  with $P\subset G=\Spin(p+1,q+1)$\ the stabilizer of an isotropic
  ray in the standard representation on $\rr^{p+1,q+1}$, cf. \cite{mrh-thesis}.

  We are now going to consider the first BGG-operators and coupling
  formulas for three 
  $\Spin(p+1,q+1)$-representations:
  for the standard representation on $\rr^{p+1,q+1}$,
  its exterior powers $\La^{k+1}\rr^{p+1,q+1},k\geq 0$\ and 
  the spin representation $\De^{p+1,q+1}$.

\subsection{BGG-operators in conformal spin geometry}
  \subsubsection{The almost Einstein scale operator $\Th_0^{\rr^{p+1,q+1}}$}
  With $T=\rr^{p+1,q+1}$\ the standard representation of
  $\Spin(p+1,q+1)$, one obtains the standard\ tractor
    bundle $\mc{T}=\G\times_P \rr^{p+1,q+1}$ together with
  its tractor metric $\mb{h}$. 

  It has a semidirect composition series
  $\mc{T}=\ce[1]\lpl\ce_a[1]\lpl\ce[-1]$\ and
   with respect to any $g\in\mc{C}$\ one obtains a decomposition
   $\mc{T}\overset{g}{\cong}
    \begin{pmatrix}
        \ce[-1] \\
        \ce_a[1] \\
        \ce[1]
    \end{pmatrix}$.

  With respect to the Levi-Civita connection
  $D$\ of $g\in\mc{C}$\ the first BGG-operator of $T$\ is
  \begin{align*}
    \Th_0^T:\ce[1]\goesto\ce_{(ab)}[2],\ 
    \si\mapsto \mb{tf}(DD\si+\PP\si),
  \end{align*}
  with  $\mb{tf}$\ denoting the trace-free part and
  round brackets symmetrization.

   If $\si\in\ker\Th_0^T$\ is non-trivial, then the
   complement of its zero set in $M$ is open dense,
   and on that set $\si$\ describes a rescaling of $g$\ to
   an Einstein metric $\si^{-2}g$. One therefore
   says that the solutions of $\Th_0^T(\si)=0$\ are
   the \emph{almost\ Einstein\ scales}\ of $\mc{C}$, cf. \cite{gover-aes}. We denote
   $\aEs(\mc{C})=\ker\Th_0^T\subset\ce[1]$.

   We will need an explicit formula for the first BGG-splitting operator
   of $\mc{T}$, cf. e.g. \cite{thomass}:
\begin{align}\label{splitStd} &L_0^{\mc{T}}:\ce[1]\goesto
\Ga(\mc{T}),\ \si\mapsto
  \begin{pmatrix} -\frac{1}{n}g^{pq}(D_{pq}\si+\PP_{pq}\si) \\ D\si \\
\si
  \end{pmatrix}.
\end{align}

  This case is particularly simple since 
  the standard tractor connection $\na^T$\
  is already the prolongation connection. So all
  solutions of $\Th_0(\si)=0$\ are normal and correspond
  to parallel standard tractors.

  \subsubsection{The conformal Killing form operator $\Th_0^{\La^{k+1}\rr^{p+1,q+1}}$}
  Now let $V=\La^{k+1}\rr^{p+1,q+1}$\ for $k\geq 1$\ be an
  exterior power of the standard representation
  and $\mc{V}=\G\times_P V$\ the associated tractor bundle.
  $\mc{V}$ has a semidirect composition series
$\mc{E}_{[a_1\cdots a_k]}[k+1]\lpl(\mc{E}_{[a_1\cdots a_{k+1}]}[k+1]
\oplus\mc{E}_{[a_1\cdots a_{k-1}]}[k-1])
\lpl\mc{E}_{[a_1\cdots a_k]}[k-1]$.

The first BGG-operator of $V$\ is
 \begin{align*}
   &\Th_0^V:\ce_{[a_1\cdots a_k]}[k+1]\goesto \mc{E} _{c[a_1\ldots a_k]}[k+1],\\
   &\si_{a_1\cdots a_k}\mapsto
D_c \si_{a_1\cdots a_k}-{D}_{[a_0}\si_{a_1\cdots a_k]}-\frac{k}{n-k+1} g_{c[a_1} g^{pq}{D}_{|p}\si_{q|a_2\cdots a_k]}
 \end{align*}
and its solutions are the conformal Killing forms.

Our coupling formulas below will employ the first BGG-splitting operator
$L_0:\HH_0\goesto \VV$,
given with respect to a $g\in\mc{C}$\ and the corresponding splitting of
the semidirect composition series. For the computation we refer to \cite{mrh-bgg,mrh-thesis}.
\begin{align}\label{L0form}
  \si\mapsto
  \begin{pmatrix}
        \begin{pmatrix}
          -\frac{1}{n(k+1)}D^pD_p\si_{a_1\cdots a_k}
          +\frac{k}{n(k+1)}D^pD_{[a_1}\si_{|p|a_2\cdots a_k]} 
          +\frac{k}{n(n-k+1)}D_{[a_1}D^p\si_{|p|a_2\cdots a_k]}
          \\
          +\frac{2k}{n}\mr{P}^p_{\; [a_1}\si_{|p|a_2\cdots a_k]}
          -\frac{1}{n}J\si_{a_1\cdots a_k}
    \end{pmatrix}
    \\
    {D}_{[a_0}\si_{a_1\cdots a_k]} \; | \;
    -\frac{1}{n-k+1}g^{pq}{D}_p\si_{qa_2\cdots a_k}
    \\
    \si_{a_1\cdots a_k}
  \end{pmatrix}.
\end{align}
Here indices between vertical bars are not skewed over.

 The prolongation connection of $\Th_0^V$\ is $\ti\na^V=\na^V+\Psi^V$
 for $\Psi^V\in\Om^1(M,\End(\mc{V})$\ as computed in \cite{mrh-bgg}.
 For our purposes it is enough to know its part of lowest homogeneity, which is,
  \begin{align}\label{prolongV}
   &\bar\Psi^V\in\Hom(\ce_{[a_1\cdots a_k]}[k+1],\mc{E}_c\t\bigl(\mc{E}_{[a_1\cdots a_{k+1}]}[k+1]
 \oplus\mc{E}_{[a_1\cdots a_{k-1}]}[k-1]\bigr)),\\\notag
   &\si\mapsto L(\si)\oplus R(\si)
 \end{align}
with
\begin{align}\label{formulasLR}
  &L(\si)=        \frac{k+1}{2}C_{[a_0a_1\ |c}^{\quad\;\ p}\si_{p|a_2\cdots a_k]}
      +\frac{(k-1)(k+1)}{2n} g_{c[a_0}C_{a_1a_2}^{\quad\;\ pq}\si_{|pq|a_3\cdots a_k]}
\\\notag       &R(\si)=\frac{(k-1)(n-2)}{2(n-k)n} C_{c[a_2}^{\quad\; pq}\si_{|pq|a_3\cdots a_k]}
-\frac{(k-1)(k-2)}{2(n-k)n}C_{[a_2a_3}^{\quad\; pq}\si_{|cpq|a_4\ldots a_k]}.
\end{align}

  \subsubsection{The twistor spinor operator $\Th_0^{\De^{p+1,q+1}}$}
  With $\De^{p+1,q+1}$\ the spin representation of $\Spin(p+1,q+1)$\ we
  form the associated spin\ tractor\ bundle\
  $\Si:=\G\times_P\De^{p+1,q+1}$.
  Recall the the (unweighted) spin bundle $\mc{S}$ of the conformal structure.
  Then $\Si$\ has a semidirect composition series $\Si=\mc{S}[\frac{1}{2}]\lpl\mc{S}[-\frac{1}{2}]$.
  
  With respect to the Levi-Civita connection $D$\ of
  a metric $g\in\mc{C}$\ the first BGG-operator of $\De^{p+1,q+1}$\ is
$    \Th_0^{\De^{p+1,q+1}}:\Ga(\mc{S}[\frac{1}{2}])\goesto\Ga( \ce_c \t \mc{S}[\frac{1}{2}])$,
  \begin{align*}
    \chi\mapsto D_c\chi+\frac{1}{n}\ga_c\dirac\chi,
  \end{align*}
  where $\ga\in\ce_c\t\End(\mc{S})$\ the Christoffel symbol of $\mc{S}$ and $\dirac\chi=g^{pq}\ga_pD_q\chi$.
The solutions of $\Th_0^{\De^{p+1,q+1}}(\chi)=0$\ are \emph{twistor spinors}.

Using $\Si\overset{g}{\cong}
    \begin{pmatrix}
        \mc{S}[-\frac{1}{2}] \\
        \mc{S}[\frac{1}{2}]
    \end{pmatrix}$ the first BGG-splitting splitting operator
is 
\begin{align}\label{twisplit} 
L_0^{\De^{p+1,q+1}}:\Ga(S[\frac{1}{2}])\goesto \Ga(\Si),\ 
   \chi\mapsto
  \begin{pmatrix} 
    \frac{\sqrt{2}}{n}\crd\chi \\ 
    \chi
  \end{pmatrix}.
\end{align}
  Also in this case, the tractor connection $\na^{\De^{p+1,q+1}}$\ coincides with the prolongation connection, which
  was employed in \cite{friedrich-conformalrelation,baum-friedrich-twistors,branson-spin,felipe-habil,mrh-thesis}).

We now relate Clifford multiplications
  and the canonical invariant pairings of the spin tractor bundle and
  and conformal spin bundle.

  \subsubsection{Clifford multiplication and invariant pairing}
  For every $g\in\mc{C}$\ we obtain identifications
  $\mc{T}\overset{g}{\cong}
    \begin{pmatrix}
        \ce[-1] \\
        \ce_a[1] \\
        \ce[1]
    \end{pmatrix},
  \Si\overset{g}{\cong}
    \begin{pmatrix}
        \mc{S}[-\frac{1}{2}] \\
        \mc{S}[\frac{1}{2}]
    \end{pmatrix}$.
  With respect to this decomposition
  the tractor metric is simply
  $\mb{h}=
  \begin{pmatrix}
    0 & 0 & 1 \\
    0 & g & 0 \\
    1 & 0 & 0
  \end{pmatrix}$
and tractor Clifford multiplication $\Ga$\ is given by
  \begin{align}\label{tcliffmult}
    &\Ga:\mc{T}\t\mc{S}\goesto \mc{S},\ 
    &\begin{pmatrix}\rh \\ \si_a \\ \si\end{pmatrix}\cdot \begin{pmatrix}\tau \\ \chi\end{pmatrix}=\begin{pmatrix}-\si_a\cdot\tau+\sqrt{2}\rh \chi \\ \si_a\cdot\chi-\sqrt{2}\si \tau\end{pmatrix}.
  \end{align}
  One easily checks that with this definition indeed
  \begin{align*}
    t_1\cdot(t_2\cdot s)+t_2\cdot(t_1\cdot s)=-2\mb{h}(t_1,t_2)\ \forall\ t_1,t_2\in\Ga(\mc{T}),s\in\Ga(\Si).
  \end{align*}

  The spin bundle $\mc{S}$\ carries a canonical pairing $\mb{b}:\mc{S}\t\mc{S}\goesto\rr$ which  is Clifford invariant in the sense that
$\mb{b}(\xi\cdot \chi,\chi')+(-1)^{p+1}\mb{b}(\chi,xi\cdot \chi')=0$\
for all $\xi\in\X(M),\chi,\chi'\in\Ga(\mc{S})$, cf. \cite{baum-pseudospin,kath}. The corresponding tractor spinor pairing is, \cite{mrh-thesis},
  \begin{align}\label{tcliffpairing}
    &\mb{B}:\Si\t\Si\goesto\rr,
    &\mb{B}\bigl(\begin{pmatrix}\tau \\ \chi\end{pmatrix},\begin{pmatrix}\tau'\\ \chi'\end{pmatrix}
\bigr)=\mb{b}(\chi,\tau')+(-1)^{p+1}\mb{b}(\chi',\tau),
  \end{align}
  which then satisfies (use \eqref{tcliffmult}), 
$\mb{B}(t\cdot X,X')+(-1)^{p}\mb{B}(X,t\cdot X')=0$
for all $t\in\Ga(\mc{T}),X,X'\in\Ga(\Si)$.

Having this background on some basic BGG-operators in conformal geometry
we can now derive a number of coupling formulas and conditions via the
method of section \ref{subseccoupl}.

\subsection{Coupling formulas}
  \subsubsection{Wedge coupling of conformal Killing forms}
  Given $s\in\Ga(\La^{k+1}\mc{T})$\ and $s'\in\Ga(\La^{k'+1}\mc{T})$\ we form
  $\mb{C}^{\wedge}(s,s'):=s\wedge s'$.
  Employing $\eqref{L0form}$\ we obtain the coupling map
\begin{align}\label{couplw1}
  &\mb{c}^{\wedge}:\ce_{[a_1\cdots a_k]}[k+1]\times \ce_{[a_1\cdots a_{k'}}[k'+1]\goesto
  \ce_{[a_1\cdots a_{k+k'+1}]}[k+k'+2]\\ \notag
  &(\si_{a_1\cdots a_k},\si'_{a_1\cdots a_{k'}})\mapsto 
  (k+1)\si_{[a_1\cdots a_k}D_{a_{k+1}}\si'_{a_{k+2}\cdots a_{k+k+1}]}\\\notag
  &\quad\quad\quad\quad\quad\quad\quad\quad+
  (-1)^{(k+1)(k'+1)}(k'+1)\si'_{[a_{1}\cdots a_{k'}}D_{a_{k'+1}}\si_{a_{k'+2}\cdots a_{k+k'+1}]}.
\end{align}

Employing Proposition \ref{propcoupling}, the
prolongation connection \eqref{prolongV}\ and some simple computatons
involving the symmetries of the Weyl curvature tensor $C$\ one shows:
\begin{prop}\label{propcw}
Assume that $\si\in\ker\Th_0^{\La^{k+1}\rr^{p+1,q}}$\ and $\si'\ker\Th_0^{\La^{k'+1}\rr^{p+1,q}}$. Then the coupled $(k+k'+1)$-form $\eta=\mb{c}^{\wedge}(\si,\si')$\ is
a conformal Killing form if and only if
\begin{align}\label{couplw2}
  (-1)^{k+1}C_{[a_1a_2\ |c}^{\quad\;\ p}\si_{p|a_3\cdots a_{k+1}}\si'_{a_{k+3}\cdots a_{k+k'+1}]}
  +\si_{[a_1\cdots a_k}C_{a_{k+1}a_{k+2}\ |c}^{\quad\quad\quad p}\si'_{p|a_{k+3}\cdots a_{k+k'+1}]}\\ \notag \overset{\odot}=0.
\end{align}
Here $\odot$\ denotes projection to the $\Spin(p,q)$-highest weight
part, which in this case are those elements in $\Ga( T^*M\t \La^k T^*M)$ with trivial
alternation and trivial trace.
\end{prop}
\begin{rem}
  Assume that $k<k'$. In the case where $k=0$\ $\si\in\ker\Th_0^{\rr^{p+1,q+1}}$\ is an almost Einstein scale and \eqref{couplw2}\ simplifies to
  \begin{align*}
    C_{[a_1a_2\ |c}^{\quad\;\ p}\si'_{p|a_3\cdots a_{k+1}]}\overset{\odot}=0,
  \end{align*}
  since $\si$\ is non-vanishing on an open dense subset. This
  agrees with Theorem 5.1 of \cite{gover-silhan-2006}. Also
  for the special case $k=1$\ \eqref{couplw2}\ and \eqref{couplc2} below
  agree with Theorem 5.4 of \cite{gover-silhan-2006}.
\end{rem}

\subsubsection{Contraction coupling of conformal Killing forms}
\def\hookin{\lrcorner}

Let now $k'>k$. We employ the tractor metric $\mb{h}$\ to form a contraction map
$\mb{C}^{\hookin}:\La^{k+1}\mc{T}\times \La^{k'+1}\mc{T}\goesto\La^{k'-k}\mc{T}$.
The coupling map is then
\begin{align}\label{couplc1}
&\mb{c}^{\hookin}:\ce_{[a_1\cdots a_k]}[k+1]\times\ce_{[a_1\cdots a_{k'}]}[k'+1]\goesto
\ce_{[a_1\cdots a_{k'-k-1}]}[k'-k]\\\notag
&(\si_{a_1\cdots a_k},\si'_{a_1\cdots a_{k'}})\mapsto 
(k+1)\si^{p_1\cdots p_k}D^q\si'_{qp_1\cdots p_k a_1\cdots a_{k'-k-1}}
\\\notag
&\quad\quad\quad\quad\quad\quad\quad\quad\ +
(n-k'+1)\si'_{p_0\cdots p_ka_1\cdots a_{k'-k-1}} D^{p_0}\si^{p_1\cdots p_k}.
\end{align}

\begin{prop}\label{propcc}
If $\si$\ and $\si'$\ are conformal Killing forms the coupled $(k'-k-1)$-form $\eta=\mb{c}(\si,\si')$\ is also a conformal Killing form
if and only if
\begin{align}\label{couplc2}
&(n-k')C^{p_0p_1}_{\quad\; qc}\si^{qp_2\cdots p_k}\si'_{p_0\cdots p_k a_1\cdots a_{k'-k-1}}\\\notag
&-(k'-1)\si^{p_1\cdots p_k}C_{cp_1}^{\quad q_1q_2}
\si'_{q_1q_2p_2\cdots p_ka_1\cdots a_{k'-k-1}} \overset{\odot}{=}0.
\end{align}
\end{prop}
\begin{rem}
  In the case where $k=0$\ and $\si$\ is an almost Einstein scale Proposition
  \ref{propcc}\ reduces to a case treated in \cite{gover-silhan-2006}, Theorem 5.1: \eqref{couplc2}\ is trivially satisfied for $k'=1$
  and simplifies to
  \begin{align}\label{couplc3}
    C_{c[a_1}^{\quad q_1q_2}
\si'_{q_1q_2 a_2\cdots a_{k'-1}]} \overset{\odot}{=}0
  \end{align}
  for $k'\geq 2$. Since the Weyl curvature tensor is skew-symmetric
  in the first two slots \eqref{couplc3}\ also holds automatically
  for $k'=2$.
\end{rem}

For our coupling formulas with twistor spinors below we assume that
the signature $(p,q)$\ is such that the spin representation $\De^{p,q}$\ 
is real, in which case also the modeling spin representation $\De^{p+1,q+1}$\ for
the spin tractor bundle is real. This avoids having to complexify the bundles $\La^k T^*M$.

\subsubsection{Twistor spinor coupling}
Let $X,X'\in\Ga(\Si)$\ and fix a $k \geq 0$.
We define an element in 
$\La^{k+1}\mc{T}\cong\La^{k+1}\mc{T}^*$\ by
\[\mb{C}^{k}(X,X')(\Ph)=\mb{B}(\Ph\cdot X,X')\ \forall\ \Ph\in\La^{k+1}\mc{T}.\]

This yields the invariant pairing from spinors to forms,
\begin{align}\label{coupltw}
&\mb{c}^k:\Ga(S[\frac{1}{2}])\times\Ga(S[\frac{1}{2}])\goesto\ce_{[a_1\cdots
a_k]}[k+1],\\\notag &(\chi,\chi')\mapsto \mb{b}(\chi,\ga_{[a_1}\cdots
\ga_{a_k]}\chi').
\end{align}

Since the prolongation connection of $\Si$\ coincides with the tractor connection
this well known map always produces a conformal Killing $k$-form from two given twistor spinors.

\subsubsection{Conformal Killing forms - twistor spinor
    coupling}
  Let $k\geq 0$. The tractor Clifford multiplication provides a map
$\mb{C}^{\ga}:\La^{k+1}\mc{T}\t\Si\goesto\Si$
and the corresponding coupling map is
\begin{align}\label{couplftw}
&\mb{c}^{\ga}:\ce_{[a_1\cdots a_k]}[k+1]\times \Ga(S[\frac{1}{2}])\goesto \Ga(S[\frac{1}{2}]),\\\notag
&\si\times\chi \mapsto
(-1)^{k+1}\frac{2(k+1)}{n}\si\cdot\dirac\chi+(d\si)\cdot\chi+\frac{k(k+1)}{(n-k+1)}(\de\si)\cdot\chi.
\end{align} Here $d\si=D_{[a_0}\si_{a_1\cdots a_k]}$\ is the exterior
derivative of $\si$\ and $\de\si=-g^{pq}D_{p}\si_{qa_2\cdots a_k}$\ is
the divergence of $\si$. The divergence term is trivial
for $k=0$, in which case $\si\in\ce[1]$.
\begin{prop}\label{proptwi1}
    Let $\chi\in\Ga(\mc{S}[\frac{1}{2}])$\ be a twistor spinor.
   If $\si\in\ce[1]$\ is an almost Einstein scale
   or a conformal Killing field $\si\in\ce_a[2]\cong\X(M)$
  then $\eta=\mb{c}^{\ga}(\si,\chi)$\ is again a twistor spinor.

  For both cases, which correspond to  $k=0,k=1$,
  one has
        \begin{align*}
          \na^{\De^{p+1,q+1}}\bigl((L_0^{\La^{k+1}\rr^{p+1,q+1}}\si)\cdot(L_0^{\De^{p+1,q+1}}\chi)\bigr)=0,
        \end{align*}
        which is equivalent to
        \begin{align*}
          L_0^{\De^{p+1,q+1}}\eta=\mb{C}^{\ga}(L_0^{\La^{k+1}\rr^{p+1,q+1}}\si,L_0^{\De^{p+1,q+1}}\chi).
        \end{align*}  
\end{prop}
\begin{proof}
  For $k=0$, which is the case where $\si\in\ce[1]$\ is an almost
  Einstein scale, both statements follow immediately since the tractor
  connection $\na^{\rr^{p+1,q+1}}$\ is already the prolongation connection
  of $\Th_0^{\rr^{p+1,q+1}}$ and all solutions are normal.
  
  For $k=1$ one has that in formula \eqref{formulasLR}\
  the term $R(\si)$\ vanishes, and
  $L(\si)=C_{a_0a_1\ c}^{\quad\;\ p}\si_{p}$.
  Denote $s=L_0^{\La^2\rr^{p+1,q+1}}\si,X=L_0^{\De^{p+1,q+1}\chi}$.
  Since $\Psi^{\De^{p+1,q+1}}=0$,
  one has, according to Proposition \ref{propTh0W},
  \begin{align*}
    \Th_0^{\De^{p+1,q+1}}\eta=-\Pi^{\De^{p+1,q+1}}_{1,\odot}\bigl(\mb{C}^{\ga}(\Psi^{\La^{2}\rr^{p+1,q+1}} s,X)\bigr)
    \overset{\odot}{=}C_{a_0a_1\ c}^{\quad\;\ p}\si_{p}\ga^{a_0a_1}\chi.
  \end{align*}
  But since $\chi$\ is a twistor spinor, $C_{a_0a_1\ c}^{\quad\;\ p}\si_{p}\ga^{a_0a_1}\chi=0$. This shows that in fact
  \begin{align*}
    \na^{\De^{p+1,q+1}}(s\cdot X)\in\Om^1(M,\mc{S}[-\frac{1}{2}]),
  \end{align*}
  and thus $\delstar(\na^{\de^{p+1,q+1}}(s\cdot X))=0$.
  By definition of $L_0^{\De^{p+1,q+1}}$\ this says
  that $L_0^{\De^{p+1,q+1}}\eta=s\cdot X$, and since
  $\eta\in\ker\Th_0^{\De^{p+1,q+1}}$\ this implies already
  $\na^{\De^{p+1,q+1}}(s\cdot X)=0$\ since $\na^{\De^{p+1,q+1}}$\
  coincides with the prolongation connection.
\end{proof}
\begin{prop}\label{proptwi2}
  Let $\chi\in\Ga(\mc{S}[\frac{1}{2}])$\ be a twistor spinor.
  For $k\geq 2$\ and $\si_{a_1\cdots a_k}\in\ce_{[a_1\cdots a_k]}[k+1]$\
  a conformal Killing form one has that
  $\eta=\mb{c}^{\ga}(\si,\chi)$\ is a twistor spinor if and
  only if  $C_{ca_1}^{\quad pq}\si_{pqa_2\cdots a_k}\ga^{a_1\cdots a_k}\chi\overset{\odot}{=}0.$
\end{prop}

\subsection{Generic twistor spinors}\label{subsecgen}
We start with an algebraic observation. Take a $k\geq 0$\ and the map
\begin{align*}
  C:\De^{p+1,q+1}\times\De^{p+1,q+1}\goesto\La^{k+1}\rr^{p+1,q+1},
\end{align*}
realized with respect to the $\Spin(p+1,q+1)$-invariant pairing
$B\in{\De^{p+1,q+1}}^*\t{\De^{p+1,q+1}}^*$.

For a fixed $X\in\De^{p+1,q+1}$\ we can form
\begin{align*}
  i_XC:\De^{p+1,q+1}\goesto\La^{k+1}\rr^{p+1,q+1},
\end{align*}
which is $G:=\Spin(p+1,q+1)_X$-invariant.
The following lemma is then easily checked.
  \begin{lem}
    Assume that $B(X,X)\not =0$. Then, after some suitable rescaling,
    one has that the map
    \begin{align*}
      P:\La^{k+1}\rr^{p+1,q+1}\goesto\La^{k+1}\rr^{p+1,q+1},\ 
      \Phi\mapsto i_XC(\Phi\cdot X)
    \end{align*}
    satisfies $P\circ P=\pm P$.

    Then $\ker P=\ker \Ga X$\ and we
    obtain a $G$-invariant decomposition
    \begin{align*}
      \La^{k+1}\rr^{p+1,q+1}=\ker\Ga X\oplus \im P.
    \end{align*}
  \end{lem}

  \begin{defn}
    We say that a twistor spinor $\chi\in\Ga(\mc{S}[\frac{1}{2}])$\ 
    is generic if $\mb{b}(\chi,\dirac\chi)\not=0$.
  \end{defn}
  It is visible from \eqref{tcliffpairing}\
  that a twistor spinor $\chi\in\Ga(\mc{S}[\frac{1}{2}])$\ is generic
  if and only if the corresponding $\na^{\De^{p+1,q+1}}$-parallel tractor
  $X=L_0^{\De^{p+1,q+1}}\chi$\ satisfies $\mb{B}(X,X)\not=0$.

  Now for a twistor spinor $\chi$\ the coupling map
   $ \X(M)\times \mc{S}[\frac{1}{2}] \goesto \mc{S}[\frac{1}{2}]$
  can be rewritten into
  \begin{align*}
   \xi\times\chi\mapsto D_{\xi}\chi-\frac{1}{4}(D_{[a}\xi_{b]})\cdot\chi+\frac{1}{2n}(D_p\xi^p)\chi.
  \end{align*}
  For a conformal Killing field $\xi\in\X(M)$
  this is just the Lie derivative of the (weighted) spinor $\chi$\ with respect to $\xi$,
  \cite{kosmann,fatibene}.
  Our algebraic observation from above together with
  Proposition \ref{proptwi1} yields:
  \begin{prop}\label{propdecomp}
       Every generic twistor spinor $\chi$\ provides a decomposition
  \begin{align}\label{equcKf}
   \cKf(\mc{C})=\cKf_{\chi}(\mc{C})\oplus\cKf_{\chi}^{\perp}(\mc{C})
  \end{align}
  of conformal Killing fields into a part which also preserves $\chi$ and
  a canonical complement.
  The projection
  \begin{align*}
    \cKf(\mc{C})\goesto\cKf_{\chi}^{\perp}(\mc{C})
  \end{align*}
  is given by
  \begin{align}\label{cKfprojection}
    \xi^a\mapsto \mb{b}(\ga^a\chi,D_{\xi}\chi-\frac{1}{4}(D_{[a}\xi_{b]})\cdot\chi+\frac{1}{2n}(D_p\xi^p)\chi).
  \end{align}
  \end{prop}
   One should regard a generic $\chi$\ on $(M,\mc{C})$
   as a refinement of the the geometric structure $\mc{C}$,
   and \eqref{equcKf}\ then
   says that there is a corresponding cKf-decomposition. This is motivated
   by the following example.

\subsubsection{Generic twistor spinors on conformal spin structures of signature $(2,3)$\ and $(3,3)$}
  Let $(M,\mc{C},\chi)$\ be a conformal spin structure of signature $(2,3)$\
  with a generic twistor spinor $\chi$.   
  Now genericity of $\chi$\ implies that  $\mc{D}_{\chi}=\ker\ga\chi$\
  is a generic rank $2$\ distribution on $M$, \cite{mrh-sag-twistors}: The subbundle $[\mc{D}_{\chi},\mc{D}_{\chi}]$ of $TM$\ spanned by Lie brackets 
  of sections of $\mc{D}_{\chi}$\ is $3$-dimensional and
  $TM=[\mc{D}_{\chi},[\mc{D}_{\chi},\mc{D}_{\chi}]]$.
  Similarly, if $\chi$\ is a generic twistor spinor an a $(3,3)$-signature
  conformal spin structure, then $\mc{D}_{\chi}=\ker\ga\chi\subset TM$\
  is a generic $3$-distribution on $M$: $[\mc{D}_{\chi},\mc{D}_{\chi}]=TM$.
  
  The conformal spin structure $\mc{C}$\ together with the generic
  twistor spinor $\chi$\ are uniquely determined by $\mc{D}\subset TM$.
  This is shown in \cite{mrh-sag-twistors} via a Fefferman-type construction which starts from any generic $2$- resp. $3$-distribution $\mc{D}\subset TM$\ and
  associates $(\mc{C}_{\mc{D}},\chi_{\mc{D}})$. Since
  there are non-flat generic distributions, this yields examples of non-flat conformal spin structures with generic twistor
  spinors.   

  It follows that the infinitesimal symmetries $\sym(\mc{D}_{\chi})$
  of the distribution $\mc{D}_{\chi}$\ are exactly those
  conformal Killing fields which preserve the twistor spinor $\chi$,
  and according to Proposition \ref{propdecomp}
  \begin{align*}
    \cKf(\mc{C})=\sym(\mc{D}_{\chi})\oplus\cKf_{\chi}^{\perp}(\mc{C}).
  \end{align*}

For signature $(2,3)$\ one obtains a particularly simple
decomposition since in that case $\cKf_{\chi}^{\perp}(\mc{C})=\aEs(\mc{C})$.
Using compositions of the coupling maps \eqref{coupltw},\eqref{couplftw}
one obtains explicit formulas: 
An almost Einstein scale $\si\in\ce[1]$\ is
mapped to the conformal Killing field
$\xi^a=\mb{b}(\ga^a\chi,-\frac{2}{5}\si\dirac\chi+(D\si)\cdot\chi)$
and
the almost Einstein
scale part of a conformal Killing field $\xi\in\X(M)$\ is
$\si=\mb{b}(\chi,D_{\xi}\chi-\frac{1}{4}(D_{[a}\xi_{b]})\cdot\chi)\in\ce[1]$.
The term $\frac{1}{10}(\de\xi)\chi$ does not appear
here since for signature $(2,3)$\ the invariant pairing $\mb{b}$\ is skew.

\end{document}